\documentclass[12pt]{article}
\usepackage{latexsym, amsmath, amsfonts, amsthm, amssymb,bbm,stmaryrd}
\usepackage{times}
\usepackage{tikz}
\usepackage[margin=3cm]{geometry}

\newcommand{\re}{\mathrm{Re}}
\newcommand{\eps}{\varepsilon}
\newcommand{\vphi}{\varphi}
\newcommand{\e}{\mathrm{e}}
\newcommand{\CC}{\mathbb{C}}
\newcommand{\RR}{\mathbb{R}}
\newcommand{\EE}{\mathbb{E}}
\newcommand{\E}{\mathbb{E}}
\newcommand{\NN}{\mathbb{N}}
\newcommand{\ZZ}{\mathbb{Z}}
\newcommand{\PP}{\mathbb{P}}
\newcommand{\cF}{\mathcal F}
\newcommand{\cL}{\mathcal L}
\newcommand{\cP}{\mathcal P}
\newcommand{\cB}{\mathcal B}
\renewcommand{\tilde}{\widetilde}

\newtheorem{theorem}{Theorem}[section]
\newtheorem{lemma}[theorem]{Lemma}
\newtheorem{proposition}[theorem]{Proposition}
\newtheorem{corollary}[theorem]{Corollary}

\theoremstyle{definition}

\theoremstyle{remark}

\begin{document}

\title{Poisson processes and a log-concave Bernstein theorem}
\author{Bo'az Klartag and Joseph Lehec}
\date{}
\maketitle

\begin{abstract}
We discuss interplays between log-concave functions and log-concave sequences.
We prove a Bernstein-type theorem, which characterizes the Laplace transform
of log-concave measures on the half-line in terms of log-concavity of the alternating
Taylor coefficients. We establish concavity inequalities for sequences
inspired by the Pr\'ekopa-Leindler and the Walkup theorems. One of our main tools
is a stochastic variational formula for the Poisson average.
\end{abstract}

\section{Introduction}

Let $\vphi \colon [0, \infty) \rightarrow \RR$ be a continuous function that is $C^{\infty}$-smooth on $(0, \infty)$. Its alternating Taylor coefficients are
\begin{equation} a_t(n) = (-1)^n \frac{\vphi^{(n)}(t)}{n!} \qquad \qquad \qquad (n \geq 0, t > 0). \label{eq_1411} \end{equation}
A function whose alternating  Taylor coefficients are non-negative is called a completely monotone function.
Bernstein's theorem  asserts that the alternating Taylor coefficients are non-negative if and only if there exists a finite, non-negative Borel measure $\mu$ on $[0, \infty)$ with
\begin{equation}  \vphi(t) = \int_0^{\infty} e^{-t x} d \mu(x) \qquad \qquad \qquad (t \geq 0). \label{eq_1027} \end{equation}
In other words, $\vphi$ is the Laplace transform of the measure $\mu$. See Widder \cite{widder} for proofs of Bernstein's theorem.
We say that the alternating Taylor coefficients are log-concave if
the sequence $(a_t(n))_{n \geq 0}$ is a {\it log-concave sequence} for any $t > 0$.
This means
that this sequence consists of non-negative numbers and for any $m,n \geq 0$
and $\lambda \in (0,1)$ such that $\lambda n + (1- \lambda) m$ is an integer,
\begin{equation} a_t(\lambda n + (1 - \lambda) m) \geq a_t(n)^{\lambda} a_t(m)^{1-\lambda}. \label{eq_926} \end{equation}
Equivalently, $a_t (n)^2 \geq a_t (n-1) a_t (n+1)$
for every $n \geq 1$, and
 the set of non-negative integers $n$
for which $a_t(n) >0$ is an interval of integers (either a finite interval of integers or an infinite one). 

\medskip A measure $\mu$ on $[0, \infty)$ is log-concave if it is either a delta measure at a certain point, or else an absolutely-continuous measure whose density $f:[0, \infty) \rightarrow \RR$
is a log-concave function. Recall that a function $f: K \rightarrow \RR$ for some convex set $K \subseteq \RR^n$ is log-concave if $f$ is non-negative and
$$ f(\lambda x + (1 - \lambda) y) \geq f(x)^{\lambda} f(y)^{1-\lambda} \qquad \qquad \qquad \text{for all} \ x,y \in K, 0 < \lambda < 1. $$

\begin{theorem}[``Log-concave Bernstein theorem''] Let $\vphi: [0, \infty) \rightarrow \RR$ be a continuous function that is $C^{\infty}$-smooth on $(0, \infty)$.
Then the alternating Taylor coefficients of $\vphi$ are log-concave if and only if $\vphi$ takes the form (\ref{eq_1027}) for a certain finite, log-concave measure $\mu$.
\label{thm_1030}
\end{theorem}

There are several known results about Laplace transforms of {\it log-convex} probability measures, rather than log-concave. In fact, Theorem \ref{thm_1030} is the log-concave analog of Hirsch's theorem, which analyzes the case where the alternating Taylor coefficients of $\vphi$ are log-convex and non-increasing. Hirsch's theorem states that this happens if and only if 
$\vphi$ takes the form (\ref{eq_1027}) for a measure $\mu$ whose density 
is non-increasing and log-convex, apart from an atom at the origin. 
See Hirsch \cite{hirsch} and
Schilling, Song and Vondra\v cek \cite[Section 11.2]{ssv}
for a precise formulation and a proof of Hirsch's theorem, and also Forst \cite{forst} and 
Sendov and Shan \cite{hristo} for related results.

\medskip In Section \ref{sec2} we prove Theorem \ref{thm_1030} by using an inversion formula for the Laplace transform
as well as the Berwald-Borell inequality \cite{berwald, Borell}. The latter inequality states that
if one divides the Mellin transform of a log-concave measure
on $[0, \infty)$ by the Gamma function, then a log-concave function is obtained. It directly implies the ``if'' part of 
Theorem \ref{thm_1030}.  In the proof of Hirsch's theorem from \cite{ssv}, the 
r\^ole of the Berwald-Borell inequality is replaced by the simpler Cauchy-Schwartz 
inequality. Theorem \ref{thm_1030} admits the following corollary:

\begin{corollary} Let $\mu$ be a finite, non-negative Borel measure on $[0, \infty)$
and let $\vphi$ be given by (\ref{eq_1027}). Then $\mu$ is log-concave
if and only if the function $\left| \vphi^{(n-1)}(t) \right|^{-1/n}$ is convex
in $t \in (0, \infty)$ for every $n \geq 1$. \label{cor_1535}
\end{corollary}

In fact, in Theorem \ref{thm_1030}
it suffices to verify that the sequence $( a_t(n) )_{n \geq 0}$ is log-concave for a sufficiently large $t$, as follows
from the following:

\begin{proposition} Let $\vphi: (0, \infty) \rightarrow \RR$ be real-analytic, and define $a_t(n)$ via (\ref{eq_1411}). Assume that $0 < r < s$ and that the sequence $( a_{s}(n) )_{n \geq 0}$ is log-concave.
Then  the sequence $( a_{r}(n) )_{n \geq 0}$ is also log-concave. \label{prop_1414}
\end{proposition}

Proposition \ref{prop_1414} is proven in Section \ref{sec3}, alongside concavity
inequalities related to log-concave sequences in  the spirit of the Walkup theorem \cite{Walkup}.
While searching for a Pr\'ekopa-Leindler type inequality for sequences, we  found  the
following:

\begin{theorem}
Let $f,g,h,k: \ZZ \rightarrow \RR$ satisfy
\[
f(x) + g(y) \leq h \left( \left \lfloor \frac{x+y}2 \right \rfloor \right)
+ k \left( \left \lceil \frac{x+y}2 \right \rceil \right) , \quad \forall x,y \in \mathbb Z,
\]
where $\lfloor x \rfloor$ is the lower integer part of $x \in \RR$
 and $\lceil x \rceil$ is the upper integer part. 
Then
\begin{equation}
\left( \sum_{x\in \ZZ} \e^{f(x)} \right)
\, \left( \sum_{x\in \ZZ} \e^{g(x)} \right)
\leq
\left( \sum_{x\in \ZZ} \e^{h(x)} \right)
\, \left( \sum_{x\in \ZZ} \e^{k(x)} \right).
\label{eq_558} \end{equation} \label{thm_603}
\end{theorem}

Our proof of Theorem \ref{thm_603} is presented in Section \ref{sec5} and it involves probabilistic techniques. It would be interesting to find a  direct proof. However, we believe that the probabilistic method is not without importance in itself,
and perhaps it is mathematically deeper than other components of this paper. The argument is based on a stochastic variational formula for the
expectation
of a given function with respect to the Poisson distribution. It is analogous to Borell's formula
from \cite{Borell3} which is concerned with the Gaussian distribution. The stochastic variational formula is
discussed in Section \ref{sec4}.

\medskip The Berwald-Borell inequality (or Theorem \ref{thm_1030}) implies that when $\mu$ is a finite, log-concave measure on $[0, \infty)$
and $k \leq \ell \leq m \leq n$  are non-negative integers with $k + n = \ell + m$,
\begin{equation}
a_t(\ell) a_t(m) - a_t(k) a_t(n) \geq 0,
\label{eq_940}
\end{equation}
where $a_{t}(n) = \int_0^{\infty} (x^n / n!) e^{-tx} d \mu(x)$
is defined via (\ref{eq_1411}) and (\ref{eq_1027}). The following
theorem shows that the left-hand side of (\ref{eq_940}) is not only non-negative, but
it is in fact a completely-monotone function of $t$:

\begin{theorem} Let $\mu$ be a finite, log-concave measure on $[0, \infty)$. Then for any non-negative integers
$k \leq \ell \leq m \leq n$  with $k + n = \ell + m$ there exists a finite, non-negative measure $\nu = \nu_{k, \ell, m, n}$ on $[0, \infty)$,
such that for any $t > 0$,
\begin{equation}  \int_0^{\infty} e^{-tx} d \nu(x) = a_t(\ell) a_t(m) - a_t(k) a_t(n), \label{eq_305}
\end{equation}
where as usual $a_{t}(n) = \int_0^{\infty} (x^n / n!) e^{-tx} d \mu(x)$
is defined via (\ref{eq_1411}) and (\ref{eq_1027}).
\label{cor_1013}
\end{theorem}

Theorem \ref{cor_1013} is proven in Section \ref{sec3}. Let us apply this theorem in a few examples. In the case where $\mu$ is an exponential measure, whose density is $t \mapsto \alpha e^{-\alpha t}$ on $[0, \infty)$,
the measures $\nu$ from Theorem \ref{cor_1013} vanish completely. In the case where $\mu$ is proportional to a Gamma distribution,
the measures $\nu$ are also proportional to Gamma distributions. When $\mu$ is the uniform measure on the interval $[1,2]$, the density of the measure $\nu = \nu_{0,1,1,2}$ from Theorem \ref{cor_1013} is depicted in Figure 1.
This log-concave density equals the convex function $(t-1) (t-2)/2$ in the interval $[2,3]$, and it equals $(t-2) (4-t)$ in $[3,4]$.

\bigskip
\begin{center}
\begin{tikzpicture}[scale = 2]
\draw [<->] (0,1.3) -- (0,0) -- (5.3,0);
\draw[samples=200, thick,  domain=2:3] plot (\x, {   (\x -2) * (\x - 1) / 2   });
\draw[samples=200, thick,  domain=3:4] plot (\x, {   (\x -2) * (4 - \x )      });

\draw (1,-0.03) -- (1,0.03);
\draw (2,-0.03) -- (2,0.03);
\draw (3,-0.03) -- (3,0.03);
\draw (4,-0.03) -- (4,0.03);
\draw (5,-0.03) -- (5,0.03);
\draw (-0.03, 1) -- (0.03, 1);
\node [below] at (1,0) {\small 1};
\node [below] at (2,0) {\small 2};
\node [below] at (3,0) {\small 3};
\node [below] at (4,0) {\small 4};
\node [below] at (5,0) {\small 5};
\node [left] at (0,1) {\small 1};

\end{tikzpicture}

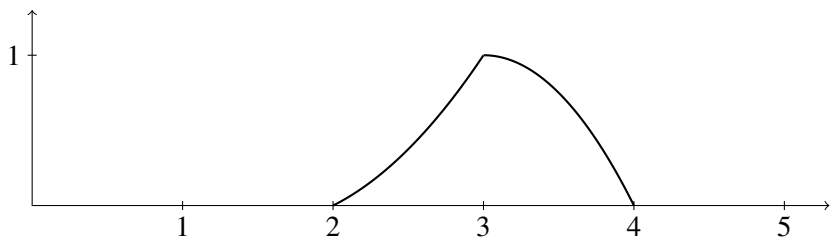
Figure 1: The density of $\nu_{0,1,1,2}$ where $\mu$ is uniform on the interval $[1,2]$.
\end{center}

\medskip We suggest that the measure $\nu$ from Corollary \ref{cor_1013} be referred to as the Berwald-Borell transform of $\mu$ with parameters $(k, \ell, m, n)$.
All Berwald-Borell transforms of log-concave measures that we have encountered so far were log-concave  themselves.
It is a curious problem to characterize the family of measures $\nu$ which could arise
as the Berwald-Borell transform of a log-concave measure on $[0, \infty)$. Such a characterization could lead
to new constraints on the moments of log-concave measures on $[0, \infty)$
beyond the constraints posed by the Berwald-Borell inequality.

\medskip
{\it Acknowledgements.} We would like 
to thank the anonymous referee for telling us about Hirsch's theorem. The first-named author was
supported in part by a grant from the European Research Council (ERC).

\section{Proof of the log-concave Bernstein theorem}
\label{sec2}

The proof of Theorem \ref{thm_1030} combines ideas of Berwald  from the 1940s
with the earlier Post inversion formula for the Laplace transform.
The ``if'' direction of Theorem \ref{thm_1030} follows from:

\begin{lemma} Let $\mu$ be a finite, log-concave measure on $[0, \infty)$. Assume that $\vphi$ is given by (\ref{eq_1027}).
Then the alternating Taylor coefficients of $\vphi$ are log-concave. \label{lem_1344}
\end{lemma}

\begin{proof} In the case where $\mu = c \, \delta_0$ for some $c \geq 0$, we have  $\vphi \equiv c$ and the alternating Taylor coefficients of $\vphi$ are trivially
	log-concave. In the case where $\mu = c \, \delta_{x_0}$ for $x_0 > 0$ we have $\vphi(t) = c \, \e^{-tx_0}$ and hence
\[
a_t(n) = \frac { c \,  \e^{-tx_0} x_0^n }{ n! } .
\]
Since $a_t (n) >0$ for every $n$ and $a_t(n+1) / a_t(n) = x_0 / (n+1)$ is
non-increasing, this is indeed a log-concave sequence.
In the case where $\mu$ has a log-concave density $f$, we denote $f_t(x) = e^{-tx} f(x)$ and observe that
\begin{equation}  \vphi^{(k)}(t) = \int_0^{\infty} (-x)^k e^{-tx} f(x) dx = \int_0^{\infty} (-x)^k f_t(x) dx \qquad \qquad (k \geq 0, t > 0). \label{eq_1321}
\end{equation}
The function $f_t$ is log-concave, and hence we may apply the Berwald-Borell inequality \cite{berwald, Borell}, see also Theorem 2.2.5 in \cite{BGVV}
or Theorem 5 in \cite{NO} for different proofs.
This inequality states that the sequence
\begin{equation}  k \rightarrow \frac{\int_0^{\infty} x^k f_t(x) dx}{k!} \qquad \qquad \qquad (k \geq 0) \label{eq_143} \end{equation}
is log-concave, completing the proof.
\end{proof}

We now turn to the proof of the ``only if'' direction of Theorem \ref{thm_1030}, which
relies on the Post inversion formula for the Laplace transform, see Feller \cite[Section VII.6]{feller}
or Widder \cite[Section VII.1]{widder}.
Suppose that $\vphi$ is continuous on $[0, \infty)$ and $C^{\infty}$-smooth on $(0, \infty)$,
and that the alternating Taylor coefficients $a_t(n)$ are log-concave.
In particular, the alternating Taylor coefficients are non-negative. We use Bernstein's theorem to conclude that
there exists a finite, non-negative Borel measure $\mu$ on $[0, \infty)$ such that (\ref{eq_1027}) holds true.
All that remains is to prove the following:
\begin{proposition} The measure $\mu$ is log-concave.
\label{prop_1344}
\end{proposition}

The proof of Proposition \ref{prop_1344} requires some preparation. First, it follows from (\ref{eq_1411}) and (\ref{eq_1027}) that
for any $R, t > 0$,
\begin{equation} \sum_{n=0}^{\lfloor R t \rfloor} t^n a_t(n) = \sum_{n=0}^{\lfloor R t \rfloor} \frac{t^n}{n!} \int_0^{\infty} x^n e^{-t x} d \mu(x)
= \int_0^{\infty} \PP \left( N_{tx} \leq R t   \right)  d \mu(x), \label{eq_1122}
\end{equation}
where $N_s$ is a Poisson random variable with parameter $s$, i.e.,
\[
\PP( N_s = n )
= e^{-s} \frac{ s^n }{ n! }, \qquad \text{for} \ n = 0, 1, 2,\ldots
\]
The random variable $N_s$ has expectation $s$ and standard-deviation $\sqrt{s}$. By the central limit theorem for the Poisson distribution
(see Feller \cite[Chapter VII]{feller} or 
Schilling, Song, and Vondra\v cek \cite[Lemma 1.1]{ssv}), for any $\alpha > 0$,
\begin{equation}  \PP(N_s \leq \alpha s) \stackrel{s \rightarrow \infty} \longrightarrow \left \{ \begin{array}{cc} 1 & \alpha > 1 \\ 1/2 & \alpha = 1 \\ 0 & \alpha < 1 \end{array} \right.
\label{eq_1123} \end{equation}
The left-hand side of (\ref{eq_1123}) is always between zero and one.
Therefore we may use the bounded convergence theorem, and conclude from (\ref{eq_1122}) that
for any $R > 0$,
\begin{equation}  \lim_{t \rightarrow \infty} \sum_{n=0}^{\lfloor R t \rfloor} t^n a_t(n)  = \mu( [0, R) ) + \frac{1}{2} \cdot \mu ( \{ R \} ). \label{eq_1141}
\end{equation}
For $t > 0$ define $g_t: [0, \infty) \rightarrow \RR$ via
$$ g_t(x) =  \left \{ \begin{array}{ll} t^{n+1} \cdot a_t(n) &  x = n /t \ \text{for some integer \ } n \geq 0 \\
t^{x+1} \cdot a_t(n)^{1-\lambda} \cdot a_t(n+1)^{\lambda} \phantom{bla} &  x = (n + \lambda) / t \ \text{for} \ \lambda \in (0,1), n \geq 0. \end{array} \right.
$$
Write $\mu_t$ for the measure on $[0, \infty)$ whose density is $g_t$.
We think about $\mu_t$ as an approximation for the discrete measure on $[0, \infty)$ that has an atom at $n/t$
of weight $t^n a_t(n)$ for any $n \geq 0$.

\begin{lemma} Assume that $\mu( (0, \infty)) > 0$. Then for any $t > 0$ the measure $\mu_t$ is log-concave on $[0, \infty)$. Moreover, if $\mu ( \{ 0 \} ) = 0$
	then for any $R > 0$,
	$$ \mu_t([0,R)) \stackrel{t \rightarrow \infty} \longrightarrow \mu( [0, R) ) + \frac{1}{2} \cdot \mu ( \{ R \} ). $$
	\label{lem_1936}
\end{lemma}

\begin{proof} Since $\mu( (0,\infty)) > 0$, for any $t >0 $ and $n \geq 0$,
	$$ a_t(n) = \int_0^{\infty} \frac{x^n}{n!} e^{-tx} d \mu(x) > 0. $$ 
	The density $g_t$ is locally-Lipschitz, and for any integer $n \geq 0$ and $x \in (n/t, (n+1)/t)$,
	$$ (\log g_t)^{\prime}(x) = \log t + t \log \frac{a_t(n+1)}{a_t(n)}. $$
	The sequence $(a_t(n))_{n \geq 0}$ is log-concave, hence $a_t(n+1)/a_t(n)$ is non-increasing in $n$. We conclude
	that $(\log g_t)^{\prime}(x)$, which exists for almost any $x > 0$, is a non-increasing function of  $x \in [0, \infty)$. This shows that 
	the locally-Lipschitz function $g_t$ is a log-concave function,
	and consequently $\mu_t$ is a
	log-concave measure.
	In particular, the density $g_t$ is unimodular, meaning that for some $x_0 \geq 0$, the function
	$g_t$ is non-decreasing in $(0, x_0)$ and non-increasing in $(x_0, \infty)$. We claim that for any $R, t > 0$
	we have the Euler-Maclaurin type bound:
	\begin{equation}
	\left| \int_0^R g_t(x) dx \, - \, \sum_{n=0}^{\lfloor R t \rfloor} \frac{1}{t} \cdot g_t \left( \frac{n}{t} \right) \right| \leq \frac{3}{t} \cdot \sup_{x > 0} g_t(x).
	\label{eq_352}
	\end{equation}
	Indeed, the sum in (\ref{eq_352}) is a Riemann sum related to the integral of $g_t$ on the interval $I = [0, \lfloor t R + 1 \rfloor / t]$.
	This Riemann sum corresponds to a partition of $I$ into segments of length $1/t$, and by unimodularity, this Riemann sum can deviate from the actual integral by at most
	$2/t \cdot \sup_{x > 0} g_t(x)$. Since the symmetric difference between $I$ and $[0,R]$ is an interval of length at most $1/t$,
	the relation (\ref{eq_352}) follows. According to (\ref{eq_352}), for any $R, t > 0$,
	\begin{equation} \left | \mu_t([0,R)) \, - \, \sum_{n=0}^{\lfloor R t \rfloor} t^n a_t(n) \right| \leq \frac{3}{t} \cdot \sup_{x > 0} g_t(x)
	= 3 \cdot \sup_{n \geq 0} t^n a_t(n).
	\label{eq_445} \end{equation}
	Next we use our assumption that $\mu ( \{ 0 \} ) = 0$ and
	also the fact that $\sup_n e^{-s} s^n / n!$ tends to zero as $s \rightarrow \infty$,
	as may be verified routinely. This shows that for $t > 0$,
	\begin{align*}
	\sup_{n \geq 0} t^n a_t(n) = \sup_{n \geq 0} \int_0^{\infty} \frac{(tx)^n}{n!} e^{-tx}  d \mu(x)
	& \leq \int_0^{\infty} \left( \sup_{n \geq 0} \frac{(tx)^n}{n!} e^{-tx} \right) d \mu(x) \stackrel{t \rightarrow \infty} \longrightarrow 0,
	\end{align*}
	where we used the dominated convergence theorem in the last passage.
	The lemma now follows from (\ref{eq_1141}) and (\ref{eq_445}).
\end{proof}

The following lemma is due to Borell,  and its proof
is contained in \cite[Lemma 3.3]{Borell2} and the last paragraph of the proof of \cite[Theorem 2.1]{Borell2}. For the reader's convenience, we include a short proof. 
For $A, B \subseteq \RR$ and $\lambda \in \RR$ we write $A + B = \{ x + y \, ; \, x \in A, y \in B \}$
and $\lambda A = \{ \lambda x \, ; \, x \in A \}$. 
We say that an interval $I \subseteq \RR$ is rational 
if it has a finite length and if its endpoints are rational numbers.

\begin{lemma} Let $\mu$ be a finite Borel measure on $\RR$ such that for 
any intervals $I,J \subseteq \RR$ and $\lambda \in (0,1)$,
	\begin{equation} \mu \left( \lambda I + (1- \lambda) J \right) \geq \mu(I)^{\lambda} \mu(J)^{1-\lambda}. \label{eq_1429} \end{equation}
	Then $\mu$ is log-concave (i.e., either $\mu = c \delta_{x_0}$ for some $c \geq 0, x_0 \in \RR$ or else $\mu$ has a log-concave density). Besides,  the conclusion remains valid if we only assume that~\eqref{eq_1429} holds for rational $I,J$ and $\lambda$. 
	\label{lem_1503}
\end{lemma}

\begin{proof} For $x \in \RR$ and $\eps > 0$ set 
	$$ f_{\eps}(x) = \frac{ \mu((x-\eps, x+\eps) )}{2 \eps} . $$
We deduce from (\ref{eq_1429}) that $f_{\eps}: \RR \rightarrow (0, \infty)$ is a log-concave function for all $\eps > 0$. 
Denote $$ f(x) = \limsup_{\eps \rightarrow 0} f_{\eps}(x) \in [0, +\infty] \qquad \qquad \textrm{for} \ x \in \RR. $$ Since 
$f_{\eps}$ is log-concave, it follows that 
for all $0 < \lambda < 1$ and $x, y \in \RR$
\begin{equation}  f(\lambda x + (1-\lambda) y) \geq f(x)^{\lambda} f(y)^{1-\lambda}, \label{eq_1713} 
\end{equation}
where in case $f(x) = \infty$ or $f(y) = \infty$ we interpret (\ref{eq_1713}) by continuity. 
By the Lebesgue differentation theorem, the function $f$ is the density of the absolutely-continuous component of the finite measure $\mu$. In particular, $f$ is integrable. Moreover, if $f(x) < \infty$ for all $x \in \RR$, then the measure $\mu$ is absolutely-continuous, and in any case, the set of all points $x \in \RR$ where $f(x) = 0$ has a zero $\mu$-measure. 

\medskip If $f(x) < \infty$ for all $x \in \RR$, then $\mu$ is absolutely-continuous 
with a log-concave density $f$, as required. 

\medskip Otherwise, there exists $x_0 \in \RR$ with $f(x_0) = +\infty$. 
Since $f(x_0) = +\infty$, 
 necessarily $f(x) = 0$ for $x \neq x_0$, as otherwise (\ref{eq_1713}) 
implies that $f$ equals $+\infty$ in an interval of positive length, 
in contradiction to the integrability of $f$. Thus $f(x) = 0$
for all $x \neq x_0$, and $\mu$ is supported at the point $\{ x_0 \}$. 
In this case necessarily $\mu = c \delta_{x_0}$ for some $c \geq 0$. 

\medskip For the second part of the lemma, observe that 
\[
 \mu((x-\eps, x+\eps) ) = \sup \left\{ \mu (I); I \subsetneq (x-\eps,x+\eps ) 
\text{ is a rational interval} \right\} . 
\]
Using this equality, one can show that if~\eqref{eq_1429} holds for rational $I,J$ and $\lambda$ 
only, then $f_\eps$ is log-concave.
We then proceed as above and conclude that $\mu$ is log-concave. 
\end{proof}

\begin{proof}[Proof of Proposition \ref{prop_1344}] We may assume 
	that $\mu ((0,\infty)) > 0$ as otherwise $\mu = c \delta_0$
	and the conclusion trivially holds. Therefore $a_t(n) > 0$
	for all $t$ and $n$. By the log-concavity of the sequence of alternating Taylor coefficients,
	\begin{equation}
	\frac{\left( \int_0^{\infty} x e^{-tx} d \mu(x) \right)^2}{
		\int_0^{\infty} (x^2/2) e^{-tx} d \mu(x) 	} = \frac{a_t(1)^2}{a_t(2)} \geq a_t(0) = \int_0^{\infty} e^{-tx} d \mu(x) \stackrel{t \rightarrow \infty}\longrightarrow \mu( \{ 0 \}). \label{eq_1456}
	\end{equation}
	For $t > 0$ write $\nu_t$ for the measure 
	on $(0, \infty)$ whose density with respect to $\mu$ equals $x \mapsto e^{-tx}$. Then by the Cauchy-Schwartz inequality,
	\begin{equation} \frac{ \left( \int_0^{\infty} x e^{-tx} d \mu(x) \right)^2 }{\int_0^{\infty} x^2 e^{-tx} d \mu(x) }
	=  \frac{ \left( \int_0^{\infty} x d \nu_t(x) \right)^2 }{\int_0^{\infty} x^2 d \nu_t(x) }
	\leq 
	\nu_t( (0, \infty)) \stackrel{t \rightarrow \infty}\longrightarrow 0.
	\label{eq_1457}
	\end{equation}
	From (\ref{eq_1456}) and (\ref{eq_1457}) we see that $\mu( \{ 0 \}) = 0$, which is required for the application of the second part of Lemma \ref{lem_1936}.
		Let $I, J \subseteq [0, \infty)$ be intervals and $0 < \lambda < 1$. Set $K = \lambda I + (1-\lambda) J$ and assume first that 
\begin{equation}
\mu(\partial I) = \mu(\partial J) = \mu(\partial K) = 0, \label{eq_1627}
\end{equation}
where $\partial I$ is the boundary of the interval $I$.
In this case, by Lemma \ref{lem_1936},
		\begin{equation} \mu(I) = \lim_{t \rightarrow \infty} \mu_t\left( I \right), \quad
		\mu(J) = \lim_{t \rightarrow \infty} \mu_t(J), \quad \mu(K) = \lim_{t \rightarrow \infty} \mu_t(K). \label{eq_1511}
		\end{equation}
		Since $\mu_t$ is log-concave, the Pr\'ekopa-Leindler inequality (see, e.g., \cite[Theorem 1.2.3]{BGVV}) implies that
		for all $t > 0$,
		\begin{equation}  \mu_t \left( K \right) \geq \mu_t(I)^{\lambda} \mu_t(J)^{1-\lambda}. \label{eq_1233}
		\end{equation}
		From (\ref{eq_1511}) and (\ref{eq_1233}) we thus deduce that 
		\begin{equation}
		\mu( \lambda I + (1-\lambda) J) \geq \mu(I)^{\lambda} \mu(J)^{1-\lambda}. \label{eq_1732}
		\end{equation}
		We know that (\ref{eq_1732}) holds
		true for any intervals $I, J \subseteq \RR$ and $0 < \lambda < 1$
		satisfying condition (\ref{eq_1627}), where $K = \lambda I + (1-\lambda) I$. Since $\mu$ is a 		     finite measure, it can only have a countable number of atoms. Hence there exists $\alpha \in		  \RR$ such that none of these atoms are congruent to $\alpha$ mod $\mathbb Q$. In other words 			by translating $\mu$, we may assume that it has no rational atoms. Then~\eqref{eq_1732} holds for all rational $I,J,\lambda$ and Lemma~\ref{lem_1503} implies that $\mu$ is a log-concave measure, as desired. 
\end{proof}

The proof of Theorem \ref{thm_1030} is complete.

\begin{proof}[Proof of Corollary \ref{cor_1535}]
	If $\mu$ is of the form $c \delta_0$ for some $c \geq 0$, then the corollary is trivial.
	Otherwise, the alternating
	Taylor coefficients $a_t(n) = (-1)^n \vphi^{(n)}(t) / n!$ are positive for 
	every $t > 0$ and $n \geq 0$. 
	The measure $\mu$ is log-concave if and only if the sequence of  
	alternating
	Taylor coefficients is log-concave for any $t > 0$, which happens if and only if
	\begin{equation} \left( \frac{\vphi^{(n)}(t)}{n!} \right)^2 - \frac{\vphi^{(n-1)}(t)}{(n-1)!} \cdot \frac{\vphi^{(n+1)}(t)}{(n+1)!} \geq 0 \qquad \qquad \qquad (n \geq 1, t > 0). \label{eq_1540}
	\end{equation}
	Denote by $b_n(t)$ the expression on the left-hand side of (\ref{eq_1540}) multiplied by $(n!)^2$. Then,
	$$ \frac{d^2}{dt^2} \left| \vphi^{(n-1)}(t) \right|^{-1/n} = \frac{d^2}{dt^2} \left( (-1)^{n-1} \cdot \vphi^{(n-1)}(t) \right)^{-1/n}
	=\frac{n+1}{n^2} \left| \vphi^{(n-1)}(t) \right|^{-(2n+1)/n} \cdot b_n(t).
	$$
	Hence $b_n(t) \geq 0$ for all $t> 0$ if and only if the function $\left| \vphi^{(n-1)}(t) \right|^{-1/n}$  is convex in $ (0, \infty)$.
\end{proof}

\section{The log-concavity measurements are completely  monotone}
\label{sec3}

In this section we prove Proposition \ref{prop_1414} and Theorem \ref{cor_1013}.
Let $k \leq \ell \leq m \leq n$ be non-negative integers with $k + n = \ell + m$ and let $\vphi$ be a continuous function on $[0, \infty)$
that is $C^{\infty}$-smooth in $(0, \infty)$.
Define
$$  c_{k, \ell, m, n}(t) = a_t(\ell) a_t(m) - a_t(k) a_t(n) = (-1)^{\ell + m} \left[ \frac{\vphi^{(\ell)}(t)}{\ell!} \frac{\vphi^{(m)}(t)}{m!} - \frac{\vphi^{(k)}(t)}{k!} \frac{\vphi^{(n)}(t)}{n!}  \right], $$
where $a_t(n) = (-1)^n \vphi^{(n)}(t) / n!$ as before. We call the functions $c_{k, \ell, m, n}: (0, \infty) \rightarrow \RR$ the log-concavity measurements  of $\vphi$. This name is justified by the following little lemma. 
For integers $a \leq b$ we write $\llbracket a, b \rrbracket = \{ n \in \ZZ \, ; \, a \leq n \leq b \}$.

\begin{lemma} Let $t > 0$. Then the sequence $(a_t(n))_{n \geq 0}$ is log-concave if and only if $c_{k, \ell, m, n}(t) \geq 0$ for all non-negative integers
$k \leq \ell \leq m \leq n$  with $k + n = \ell + m$. \label{lem_1036}
\end{lemma}
\begin{proof} Assume first that the sequence  $(a_t(n))_{n \geq 0}$ is log-concave. Fix non-negative integers $k \leq \ell \leq m \leq n$ with $k + n = \ell + m$, and let us prove that $c_{k,\ell, m, n}(t) \geq 0$. 
	This is obvious in the case where $k=\ell$ and $m=n$. Otherwise, set 
	$\lambda = (\ell-k) / (n - k) \in (0,1)$. 
	Then $\ell = \lambda n + (1-\lambda) k$ and $m = \lambda k + (1-\lambda) n$.
	According to (\ref{eq_926}), 
$$ a_t(\ell) \geq a_t(k)^{1-\lambda} a_t(n)^{\lambda} \qquad \text{and} \qquad a_t(m) \geq a_t(k)^{\lambda} a_t(n)^{1-\lambda}. $$	
By multiplying these two inequalities, we conclude that $c_{k, \ell, m, n}(t) \geq 0$. For the other direction, assume that the log-concavity measurements are non-negative. In particular, for any $n \geq 1$,
 $$ 0 \leq c_{n-1, n, n, n+1}(t) = a_t(n)^2 - a_t(n-1) a_t(n+1). $$
In remains to show that the set of 	
non-negative integers $n$ with $a_t(n) > 0$ is an interval of integers. Assume that $k \leq n$ satisfy $a_t(k) > 0$ and $a_t(n) > 0$.
Given any integer $\ell \in \llbracket k, n \rrbracket$, we set $m = n+k - \ell \in \llbracket k, n \rrbracket$. By the non-negativity of the log-concavity measurements,
\begin{equation} 
 a_t(\ell) a_t(m) \geq a_t(k) a_t(n) > 0. \label{eq_1037} \end{equation}
 From (\ref{eq_1037})  we deduce that $a_t(\ell) > 0$ for any integer $\ell \in \llbracket k, n \rrbracket$, as desired. \end{proof}

\begin{proposition} The derivative of each log-concavity measurement
	is a linear combination with constant, non-positive coefficients of
	a finite number of log-concavity measurements. \label{lem_1139}
\end{proposition}

\begin{proof}
	Differentiating (\ref{eq_1411}) we  obtain
	$$ \frac{d}{dt} a_t(n) = -(n+1) a_t(n+1) \qquad \qquad \qquad (t > 0, n \geq 0). $$
	Abbreviate $b_j = a_t(j)$. Then,
	\begin{equation} -c_{k, \ell, m, n}^{\prime}(t) = (\ell + 1) b_{\ell+1} b_{m} + (m + 1) b_{\ell} b_{m+1} - (k + 1) b_{k+1} b_{n} - (n+1) b_{k} b_{n+1}. \label{eq_1127}
	\end{equation}
	Assume first that $\ell < m$. In this case we may rewrite the right-hand side of (\ref{eq_1127}) as
	$$ (k+1) \left[ b_{\ell+1} b_m - b_{k+1} b_n \right] + (\ell-k) \left[ b_{\ell+1} b_m - b_{k} b_{n+1} \right] + (m+1) \left[ b_{\ell} b_{m+1} - b_{k} b_{n+1} \right]. $$
	Therefore, in the case $\ell < m$, we expressed $-c_{k, \ell, m, n}^{\prime}(t)$ as a linear combination with non-negative coefficients of three log-concavity measurements.
	From now on, we consider the case $\ell = m$.
	If $k = \ell$, then necessarily $n = m$ and the log-concavity measurement  $c_{k, \ell, m, n}(t)$ vanishes.
	If $k < \ell$, then necessarily $m < n$ and we rewrite the right-hand side of (\ref{eq_1127}) as
	$$ (k+1) \left[ b_{\ell} b_{m+1} - b_{k+1} b_n \right] + (n+1) \left[ b_{\ell} b_{m+1} - b_{k} b_{n+1} \right]. $$
	Consequently, in the case $\ell = m$, we may express $-c_{k, \ell, m, n}^{\prime}(t)$ as a linear combination with non-negative coefficients of two log-concavity measurements.
	The proof is complete.
\end{proof}

\begin{corollary} Let $t > 0$ be such that $(a_{t}(n))_{n \geq 0}$ is a log-concave sequence.
	Assume that $k \leq \ell \leq m \leq n$ are non-negative integers with $k + n = \ell + m$. Abbreviate
	$f(t) = c_{k, \ell, m, n}(t)$. Then for all $j \geq 0$,
	$$ (-1)^j f^{(j)}(t) \geq 0. $$ \label{cor_1140}
\end{corollary}

\begin{proof} Any log-concavity measurement is non-negative at any $t > 0$.
	It follows from Proposition \ref{lem_1139} that $(-1)^j f^{(j)}(t)$ is a finite linear combination with non-negative coefficients
	of certain log-concavity measurements. Therefore $(-1)^j f^{(j)}(t) \geq 0$.
\end{proof}

\begin{proof}[Proof of Proposition \ref{prop_1414}] Write $A \subseteq (0, \infty)$
	for the set of all $t > 0$ for which $(a_t(n))_{n \geq 0}$ is a log-concave
	sequence. Since $\vphi$ is $C^{\infty}$-smooth, the set $A$ is closed in $(0, \infty)$.
	From our assumption, $s \in A$. Define
	$$ t_0 = \inf \left \{ t > 0 \, ; \, [t, s] \subseteq A \right \}. $$
	Then $t_0 \leq s$. Our goal is to prove that $t_0 = 0$.
	Assume by contradiction that $t_0 > 0$. Since $A$ is a closed set, necessarily $t_0 \in A$.
	Since $\vphi$ is real-analytic, the Taylor series of $\vphi$
	converges to $\vphi$ in $(t_0 - \eps, t_0 + \eps)$ for a certain $\eps > 0$. Assume that $k \leq \ell \leq m \leq n$ are non-negative integers with $k + n = \ell + m$.
	Then also the Taylor series of $f(t) = c_{k, \ell, m, n}(t)$ converges to $f$ in the same interval $(t_0 - \eps, t_0 + \eps)$.
	From Corollary \ref{cor_1140} we thus deduce that for all $t \in (t_0 - \eps, t_0]$,
	$$ c_{k, \ell, m, n}(t) \geq 0. $$
	Consequently, $(t_0 - \eps, t_0] \subseteq A$, in contradiction to the definition of $t_0$.
\end{proof}

We proceed with yet another proof of Proposition \ref{prop_1414}, which is more in the spirit of the Walkup theorem which we shall now recall:

\begin{theorem}[Walkup theorem \cite{NO, Walkup}]
	If $(a_n)_{n \geq 0}$ and $(b_n)_{n \geq 0}$ are log-concave sequences, then the sequence $(c_n)_{n \geq 0}$
	given by
	\[
	c_n = \sum_{k=0}^n {n\choose k} a_k b_{n-k} , \qquad \qquad (n \geq 0)
	\]
	is also log-concave. \label{thm_walkup}
\end{theorem}

By Taylor's theorem, whenever $0 < s < t$,
$$ (t -s)^k a_s(k) = \sum_{n = k}^{\infty}  {n\choose k} (t - s)^n a_t(n), $$
assuming that $\vphi$ is real-analytic and that the Taylor series of $\vphi$ at $t$ converges in $(s - \eps, t+\eps)$ for some $\eps > 0$.
We conclude that Proposition \ref{prop_1414} is equivalent
to the following Walkup-type result:

\begin{proposition}
	If $(a_k)_{k \geq 0}$ is a log-concave sequence then the sequence $(c_k)_{k \geq 0}$
	defined by
	\[
	c_k = \sum_{n =k}^{\infty} {n\choose k} a_n , \qquad \qquad (k \geq 0)
	\]
	is log-concave as well. \label{prop_1710}
\end{proposition}

We do not know of a formal derivation of Theorem \ref{thm_walkup}
from Proposition \ref{prop_1710} or vice versa, yet we provide a direct proof of Proposition \ref{prop_1710} which bears
some similarity to the proof of Walkup's theorem and the Borell-Berwald inequality given in \cite{NO, Walkup}. 
We begin the direct proof of Proposition \ref{prop_1710} with the following:

\begin{lemma}
	Let $(a_n)_{n \geq 0}$ be a log-concave sequence. Then for
	every non-negative integers $k$ and $l$ we have
	\begin{equation}
	\sum_{n \geq 0} {n\choose k}{l-n\choose k} a_n a_{l-n}
	\geq \sum_{n \geq 0} {n\choose k-1}{l-n\choose k+1} a_n a_{l-n},
	\label{eq_201} \end{equation}
	where here we set ${n \choose k} = 0$ in the case where $k > n$ or $k < 0$ or $n < 0$.
\end{lemma}
\begin{proof} Inequality (\ref{eq_201}) holds trivially if $2k > l$. We may 
	thus assume that $2 k \leq l$. 	Let $U$ be a random
	subset of cardinality $2k+1$ of $\{1,\dotsc,l+1\}$ chosen
	uniformly. Let
	$X_1, \dotsc , X_{2k+1}$ be the elements of $U$ in increasing order.
	Observe that the law of $X_{k+1}$ is given by
	\[
	\PP ( X_{k+1} = n+1 ) = \frac { {n \choose k} {l-n \choose k}}{{ l+1 \choose 2k+1 }} , \qquad (n \geq 0).
	\]
	Therefore
	\[
	\sum_{n} {n\choose k}{l-n\choose k} a_n a_{l-n} = { l+1 \choose 2k+1 } \, \E [  f ( X_{k+1} ) ]
	\]
	where $f$ is the function given by
	\[
	f ( n ) = a_{n-1} a_{l+1-n} , \quad \forall n,
	\]
	and we set $a_k = 0$ for $k < 0$. In a similar way
	\[
	\sum_{n} {n\choose k-1}{l-n\choose k+1} a_n a_{l-n} = { l+1 \choose 2k+1 } \, \E [  f ( X_{k} ) ].
	\]
	Hence the desired inequality boils down to 
	\[
	\E [  f ( X_{k} ) ]
	\leq \E [ f ( X_{k+1} ) ] .
	\]
	By Fubini it suffices to prove that
	\[
	\PP (   f ( X_{k} ) > t )
	\leq \PP (  f ( X_{k+1} ) > t  ) \qquad \qquad \forall t \geq 0. 
	\]
	The sequence $(f(n))_{n \geq 0}$ is log-concave, since it is the pointwise product of two log-concave sequences.
	It also satisfies
	$f ( l+2 - n ) = f(n)$ for all $n$. The crucial observation is that 
	because of
	the log--concavity and symmetry of the sequence $(f(n)))_{n \geq 0}$, the farther $n$
is	from the midpoint $(l+2)/2$, the smaller $f(n)$.
	Therefore the level set $\{f>t\}$ is either empty
	or else an interval of the form $\llbracket n , l+2-n\rrbracket$
	for some integer $n \leq (l+2)/2$.
	Hence it suffices to prove that for every such $n$,
	\begin{equation}
	\PP (  X_k \in \llbracket n , l+2-n\rrbracket )
	\leq \PP ( X_{k+1} \in \llbracket n, l+2-n\rrbracket ). \label{eq_221}
	\end{equation}
	Intuitively, since $X_{k+1}$ is the middle
	element of $U$, it is more likely to be close to the center of the interval
	$\llbracket 1 , l+1\rrbracket$ than any other element.
	More precisely, since $X_k\leq X_{k+1}$,
	\[
	\begin{split}
	& \PP ( X_k \in \llbracket n,l+2-n \rrbracket ) - \PP ( X_{k+1} \in \llbracket n,l+2-n\rrbracket ) \\
	& = \PP ( X_k \leq l+2-n ; \; X_{k+1} > l+2-n ) - \PP ( X_k < n ;\;  X_{k+1} \geq n ) \\
	& = \frac{ {l+2-n \choose k} {n-1 \choose k+1} }{{l+1 \choose 2k+1}}
	- \frac{ {n-1\choose k} {l+2-n \choose k+1} }{{ l+1 \choose 2k+1 }} .
	\end{split}
	\]
	In order to complete the proof of (\ref{eq_221}) we need to show that this expression is non-positive, assuming 
	that $k \leq l/2$ and $n \leq (l +2) / 2$.  Note that 
	\begin{equation}
	\frac{ {l+2-n \choose k} {n-1 \choose k+1} }{  {n-1\choose k} {l+2-n \choose k+1}  }
	= \frac{(n-1-k)! (l-n-k+1)!}{(l+2-n-k)! (n-k-2)!} =
	\frac{ (n-1) - k}{(l+2-n) - k}.
	\label{eq_214} \end{equation}
	We need to show that the expression in (\ref{eq_214}) is at most one. 
	The denominator in (\ref{eq_214}) is positive, as 
	$$ l + 2 - n - k = 1 + [(l+2)/2 - n] + (l/2 - k) \geq 1. $$
	The numerator in (\ref{eq_214}) is smaller than the denominator, as $n-1 < l + 2 - n$. 
	Hence the expression in (\ref{eq_214}) is at most one, completing the proof of the lemma.
\end{proof}
\begin{proof}[Direct proof of Proposition \ref{prop_1710}:]
	The set of all $k$ with $c_k > 0$ 
	is clearly the interval of integers
	$\{ k \geq 0 \, ; \, \exists n \geq k, a_n > 0 \}$. Let $k \geq 0$ be an integer. We need to prove that
	\[ c_k^2 = 
	\sum_{n,m} {n\choose k}{m\choose k} a_n a_m
	\geq \sum_{n,m} {n\choose k-1}{m\choose k+1} a_n a_m = c_{k-1} c_{k+1}.
	\]
	By grouping the terms according to the value of $n+m$
	we see that it suffices to prove that
	for any $l, k \geq 0$,
	\[
	\sum_{n} {n\choose k}{l-n\choose k} a_n a_{l-n}
	\geq \sum_{n} {n\choose k-1}{l-n\choose k+1} a_n a_{l-n}.
	\]
	This is, however, precisely the statement of the previous lemma.
\end{proof}

When $\mu$ is a finite, log-concave measure on $[0, \infty)$, it is well-known (e.g., \cite{BGVV}) that $\mu([t, \infty)) \leq A \e^{-B t}$ for all $t > 0$, where $A, B > 0$ depend only on $\mu$.
It follows that the Laplace transform $\vphi$ defined in (\ref{eq_1027}) is holomorphic in  $\{ t \in \CC \, ; \, \re(t) > -B \}$
for some $B > 0$ depending on $\mu$.

\begin{proof}[Proof of Theorem \ref{cor_1013}] 
	By Theorem \ref{thm_1030}, the alternating Taylor coefficients sequence $(a_t(n))_{n \geq 0}$ is log-concave
	for any $t > 0$. From Corollary \ref{cor_1140} we thus learn that 
	$$ f(t) = c_{k, \ell, m, n}(t) = a_t(\ell) a_t(m) - a_t(k) a_t(n)  $$
	satisfies $(-1)^j f^{(j)}(t) \geq 0$ for any $t > 0$ and $j \geq 0$. 
	The function $f$ is real-analytic in a neighborhood of $[0, \infty)$
	and in particular it is continuous in $[0, \infty)$.
		The function $f$ is thus completely-monotone, and according to the Bernstein theorem,
	there exists a finite, non-negative measure $\nu$ for which (\ref{eq_305}) holds true.
\end{proof}

We may rewrite conclusion (\ref{eq_305}) of Theorem \ref{cor_1013} as follows: For any $t > 0$,
$$ \int_0^{\infty} \frac{x^\ell}{\ell!} e^{-tx} d \mu(x)
\int_0^{\infty} \frac{x^m}{m!} e^{-tx} d \mu(x) - \int_0^{\infty} \frac{x^k}{k!} e^{-tx} d \mu(x)
\int_0^{\infty} \frac{x^n}{n!} e^{-tx} d \mu(x) = \int_0^{\infty} e^{-tx} d \nu(x). $$
Let us now consider the Fourier transform
$$ F_{\mu}(t) = \int_0^{\infty} e^{-i tx} d \mu(x) \qquad \qquad \qquad (t \in \RR). $$
By analytic continuation, Theorem \ref{cor_1013} immediately implies the following:

\begin{proposition} Let $\mu$ be a finite, log-concave measure on $[0, \infty)$. Then for any non-negative integers
	$k \leq \ell \leq m \leq n$  with $k + n = \ell + m$ there exists a finite, non-negative measure $\nu = \nu_{k, \ell, m, n}$ on $[0, \infty)$,
	such that for any $t > 0$,
	$$ \frac{F_{\mu}^{(\ell)}(t)}{\ell!} \frac{F_{\mu}^{(m)}(t)}{m!} - \frac{F_{\mu}^{(k)}(t)}{k!} \frac{F_{\mu}^{(n)}(t)}{n!}
	= (-i)^{\ell + m} F_{\nu}(t). $$
	\label{prop_1013}
\end{proposition}

\begin{corollary} Let $\mu, k, \ell, m, n, \nu$ be as in Theorem \ref{cor_1013}.
	Write $P_j(\mu)$ for the measure whose density with respect to $\mu$ is $x \mapsto x^j / j!$.
	Write $E_t(\mu)$ for the measure whose density with respect to $\mu$ is $x \mapsto \exp(-t x)$.
	Then,
	\begin{enumerate}
		\item[(i)]  We have $\displaystyle \nu = P_\ell(\mu) * P_m(\mu) - P_k(\mu) * P_n(\mu)$ where $*$ stands for convolution.
		\item[(ii)] For any $t > 0$, the measure $E_t(\nu)$ is the Berwald-Borell transform of $E_t(\mu)$ with the same parameters $(k, \ell, m, n)$.
		The same holds for any $t \in \RR$ for which $E_t(\mu)$ is a finite measure.
	\end{enumerate}
	\label{prop_1730}
\end{corollary}

\begin{proof} We note that $F_{\mu}^{(j)} / j! = (-i)^{j} \cdot F_{P_{j}(\mu)}$. Proposition
	\ref{prop_1013} thus shows that
	\begin{equation}  F_{P_{\ell}(\mu)} F_{P_{m}(\mu)} - F_{P_{k}(\mu)} F_{P_{n}(\mu)} = F_{\nu}. \label{eq_1740} \end{equation}
	The Fourier transform maps products to convolutions. Conclusion (i) therefore follows from (\ref{eq_1740}).
	Conclusion (ii) follows immediately from the definitions.
\end{proof}

\section{Borell-type formula for the Poisson measure}
\label{sec4}

%
%
In~\cite{Borell3}, Borell gave a new proof of the Pr\'ekopa-Leindler
inequality based on the following stochastic variational formula.
Let $\gamma_n$ be the standard Gaussian measure on $\RR^n$.
Given a standard $n$-dimensional Brownian motion $(B_t)_{t \geq 0}$ and a
bounded, measurable function $f\colon \mathbb R^n\to\mathbb R$ we have
\begin{equation}
\label{borell}
\log \left( \int_{\mathbb R^n} \e^f \,  d\gamma_n \right) = \sup_u \left\{
\EE  \left[  f \left( B_1 + \int_0^1 u_s \, ds \right) -
\frac{1}{2} \int_0^1 \vert u_s \vert^2 \, ds \right] \right\} ,
\end{equation}
where the supremum is taken over all bounded stochastic processes $u$
which are adapted to the Brownian filtration, i.e.
$u_t$ is measurable with respect to the $\sigma$-field generated by $\{ B_s ; \; s\leq t \}$
for all $t\in [0,1]$.

\medskip
In this section we give a discrete
version of Borell's formula in which
the Gaussian measure and the Brownian motion
are replaced by the Poisson distribution and the
Poisson process, respectively. In the following
section we shall apply our formula in order to deduce a discrete version
of the Pr\'ekopa-Leindler inequality.
We begin with some background on counting processes with stochastic intensities.
Let $T > 0$ be a fixed number, denote $\RR_+ = [0, \infty)$,
and let $(\Omega, \cF, \PP)$ be a probability space on which our random variables will be defined.

\medskip
Throughout this section, we let $N$ be a Poisson point 
process on $[0, T] \times \RR_+ \subseteq \RR^2$
with intensity measure equal to the Lebesgue measure $\cL$.
In particular $N(F)$ is a Poisson random variable with parameter $\cL(F)$
for any Borel set $F \subseteq
[0, T] \times \RR_+ $. 
For a Borel subset $E \subseteq [0, T] \times \RR_+$ we write
$\cF_E$ for the $\sigma$-field generated by the random variables
$$ \left \{ N ( F ) ;\;  F\text{ is a Borel set},\, F\subseteq E \right \}. $$
For $t \in [0,T]$ we set $\mathcal F_t = \mathcal F_{[0,t]\times \RR_+}$.
This
defines a filtration of $\Omega$. 
Recall that a stochastic process $(\lambda_t)_{0 \leq t \leq T}$ is called \emph{predictable} if,
as a function of $t \in [0, T]$ and $\omega \in \Omega$, it is measurable
with respect to the $\sigma$-field $\cP$  generated by the sets
\[
\{ \, (s,t] \times A \, ; \,  s\leq t \leq T , \, A \in \mathcal F_s \, \}.
\]
This is slightly more restrictive than being adapted, i.e., when  $\lambda_t$ is measurable with respect to $\cF_t$.
We have the following standard fact: if a process
is left-continuous and adapted, then it is predictable.

\medskip 
Given a predictable, bounded, non-negative stochastic process  $(\lambda_t)_{0 \leq t \leq T}$
we define the associated counting process $(X^{\lambda}_t)_{0 \leq t \leq T}$ via 
\begin{equation}\label{eq:defXlambda}
X^\lambda_t = N ( \{ (s,u) \in [0, T] \times \RR^+ \, ; \, s < t ,\, u \leq \lambda_s \} ) .
\end{equation}
In other words $X^\lambda_t$ is the number of atoms
of $N$ which lie below the curve $\{ (s,\lambda_s)\colon s \in [0,t)\}$. 
The counting process $X^{\lambda}$ defined via (\ref{eq:defXlambda}) is clearly adapted, non-decreasing, integer-valued and left-continuous. Note that given $M>0$, with probability one the process $N$ has only finitely many atoms in the box $[0,T]\times [0,M]$ and no two of those lie on the same vertical line $\{t\}\times [0,M]$. Thus, with probability one the process $X^\lambda$ has finitely many jumps, all of size $1$.
We sometimes refer to $(\lambda_t)$ as the stochastic intensity of the counting process $(X^{\lambda}_t)$, and to the jumps of $X^{\lambda}$ as atoms.

\begin{lemma} For every non-negative predictable
	process $(H_t)_{0 \leq t \leq T}$ we have
	\begin{equation}
	\EE\left[ \int_0^T H_t \, X^{\lambda} ( dt ) \right]
	= \EE\left[ \int_0^T H_t \lambda_t \, dt \right],
	\label{eq_1521} \end{equation}
	where the integral on the left-hand side is a Riemann-Stieltjes integral, i.e., here it is a sum of the values of $H_t$ at the atoms of  $X^{\lambda}$.
	\label{lem_1532}
\end{lemma}

The proof of the technical Lemma \ref{lem_1532} is deferred to the
appendix.
Equation (\ref{eq_1521}) is frequently taken as the definition of
a counting process with stochastic intensity $\lambda$. The process
$$ \tilde{X}^{\lambda}_t = X^{\lambda}_t - \int_0^t \lambda_s ds \qquad \qquad (0 \leq t \leq T) $$
is called the \emph{compensated} process.
By Lemma \ref{lem_1532} it has the property that for every bounded, predictable process $(H_t)_{0 \leq t \leq T}$,
the process $$ \left( \int_0^t H_s \tilde X^{\lambda} (ds) \right)_{0 \leq t \leq T} $$
is a martingale. We are now in a position to state the analogue
of Borell's formula for the Poisson measure. 
In the following theorem $\pi_T$ denotes the
Poisson measure with  parameter $T$, i.e.,
\[
\pi_T ( n  ) = 
\frac{ T^n }{ n! } \e^{-T}  \qquad \text{for} \ n \in \NN = \{ 0,1,2,\ldots \}
\]
where we abbreviate $\pi_T(n) = \pi_T ( \{ n \} )$.

\begin{theorem}\label{thm:borell-poisson}
	Let $f \colon \NN \to \RR$ be bounded and let $T > 0$. Then we have
	\begin{equation}\label{eq:borell-poisson}
	\log \left( \int_\NN \e^f \, d\pi_T \right)
	= \sup_{\lambda} \left\{ \EE \left[
	f ( X^\lambda_T ) - \int_0^T \left(\lambda_t \log \lambda_t  - \lambda_t + 1 \right) \, dt
	\right] \right\} ,
	\end{equation}
	where the supremum is taken over all bounded, non-negative, predictable processes $(\lambda_t)_{0 \leq t \leq T}$,
	and $(X^\lambda_t)_{0 \leq t \leq T}$ is the associated counting process,
	defined by~\eqref{eq:defXlambda}. Moreover the
	supremum is actually a maximum.
\end{theorem}
\begin{proof}
	Let $(P_t)_{t \geq 0}$ be the Poisson semigroup: For every $g\colon \NN\to\RR$
	\[
	P_t g (x) =  \sum_{n\in\NN} g ( x +  n ) \, \pi_t (n) .
	\]
	We shall show that for every predictable, non-negative, bounded process $(\lambda_t)$ we have
	\begin{equation} \label{eq:Borell-goal}
	\log P_T ( e^f ) ( 0 ) \geq \EE \left[
	f ( X^\lambda_T )
	- \int_0^T \left(\lambda_t \log \lambda_t  - \lambda_t + 1 \right) \, dt \right], 
	\end{equation}
	with equality if $\lambda$ is chosen appropriately.
	Let us start with the inequality.
	Note that for every $g: \NN \rightarrow \RR$ and $t \geq 0$,
	\[
	\partial_t P_t g = \partial_x P_t g
	\]
	where $\partial_x g (x) = g(x+1) - g(x)$ denotes the discrete gradient.
	Letting
	\[
	F(t,x) = \log P_{T-t}( \e^{f} ) (x)
	\]
	we obtain
	\[
	\partial_t F = - \e^{\partial_x F } + 1 .
	\]
	Let $\lambda$ be a predictable, non-negative, bounded process and let
	\[
	M_t = F ( t , X^{\lambda}_t )
	- \int_0^t \left( \lambda_s \log \lambda_s - \lambda_s +1 \right) \, ds .
	\] Almost surely, the process $(M_t)_{0 \leq t \leq T}$ is a piecewise absolutely-continuous function in $t$.
	Thus the distributional derivative of the function $t \mapsto M_t$ is almost-surely the sum of 
	an integrable function on $[0, T]$ and finitely many atoms. Namely, for any fixed $ t \in [0, T]$,
	\begin{equation}
	M_{t} - M_0 = \int_0^{t} \partial_x F( s, X^{\lambda}_{s} ) \, X^\lambda (ds)
	- \int_0^{t} \left( \e^{ \partial_x F(s,X^{ \lambda }_{s} ) }
	+ \lambda_s \log \lambda_s - \lambda_s \right) \, ds.
	\label{eq_543} \end{equation}
	Setting $\alpha_t = \partial_x F(t,X^{\lambda}_{t} )$,
	we may rewrite  (\ref{eq_543}) as follows:
	\begin{equation}
	M_{t} - M_0 = \int_0^t \alpha_s \, \tilde X^\lambda (ds)
	- \int_0^{t} \left( \e^{\alpha_s} + \lambda_s \log \lambda_s - \lambda_s - \alpha_s \lambda_s  \right) \, ds.
	\label{eq_548} \end{equation}
	Recall that $\tilde X^\lambda_t = X^\lambda_t - \int_0^t \lambda_s \,ds$ is the compensated process.
	Note that $F(t,x)$ is continuous in $t$ and that $(X^{\lambda}_{t})$ is left-continuous in $t$.
	Thus $(\alpha_t)$ is left continuous. Since $(\alpha_t)$ is also adapted,
	it is predictable.
	Moreover both $(\alpha_t)$ and $(\lambda_t)$ are bounded. Consequently the first
	summand on the right-hand side of (\ref{eq_548}) is a martingale.
	Furthermore, since
	\begin{equation}\label{eq:Legendre}
	\e^x + y \log y - y - xy \geq 0    \qquad \forall x\in \RR , \, y \in \RR_+
	\end{equation}
	the second integral on the right-hand side of (\ref{eq_548}) is non-negative. Therefore
	$(M_t)_{0 \leq t \leq T}$ is a supermartingale.
	In particular $M_0 \geq \EE [ M_T ]$,
	which is the desired inequality (\ref{eq:Borell-goal}).
	
	\medskip 
	There is equality in~\eqref{eq:Legendre}
	if $\e^x = y$. Hence if $\lambda$ is such that
	\begin{equation}\label{eq:EDS}
	\lambda_t = \e^{ \partial_x F(t,X^{\lambda}_{t} ) } ,
	\end{equation}
	for almost every $t$ and with probability one,
	then $M$ is a martingale and we have equality in~\eqref{eq:Borell-goal}.
	Note that the function $\e^{ \partial_x F (t,x)}$
	is continuous in $t$ and bounded.
		In Lemma~\ref{lem:tech} below
we prove 
	that under these conditions, a solution to~\eqref{eq:EDS}
	does indeed exist, which concludes the proof
	of the theorem.
\end{proof}
{\it Remarks.} \begin{enumerate}
	\item
	It is also possible to prove Theorem \ref{thm:borell-poisson} by using 
	the Girsanov change of measure formula for counting processes.
	The argument presented here has the advantage
	of being self-contained.
	\item Theorem \ref{thm:borell-poisson}
	can be generalized in several ways.
	Firstly, up to some technical details,
	the argument should work just the same for a function $f$ that depends
	on the whole trajectory of the process rather than just the terminal point.
	On the left-hand side, the Poisson distribution should then be
	replaced by the law of the Poisson process of intensity $1$ on $[0,T]$.
	In the Gaussian case, this pathspace version of the formula is
	known as the Bou\'e-Dupuis formula, see~\cite{boue-dupuis}. Then one
	can also replace the interval $[0,T]$ equipped with the Lebesgue
	measure by a more general measure space, leading to a Borell-type
	formula for Poisson point processes.
	This program was actually already carried out
	by Budhiraja, Dupuis and Maroulas in~\cite{BDM}. Theorem~\ref{thm:borell-poisson}
        is thus a particular case of their main result. However, their argument is a lot more 
	intricate than the above proof. 
	\item
	A dual version of Borell's formula involving relative
	entropy was proved by the second-named author in~\cite{lehec}.
	This can be done in the Poisson case too. The formula then reads: If $\mu$
	is a probability measure on $\NN$ whose density with respect to $\pi_T$
	is bounded away from $0$ and $+\infty$, then the relative entropy
	of $\mu$ with respect to $\pi_T$ satisfies
	\[
	\mathrm H ( \mu \mid \pi_T ) = \inf_{\lambda} \left\{
	\EE \left[ \int_0^T ( \lambda_t \log \lambda_t - \lambda_t +1 ) \, dt \right] \right\} ,
	\]
	where the infimum runs over all non-negative, bounded, predictable processes
	$\lambda$ such that $X^\lambda_T$ has law $\mu$. This follows  from 
	the representation formula  (\ref{eq:borell-poisson}) and  the Gibbs variational principle, which is the fact that the functionals $\nu \mapsto H(\nu | \pi_T)$
	and $f \mapsto \log \int e^f d \pi_T$ are Legendre-Fenchel conjugates with respect to the pairing $(f, \nu) \mapsto \int f d \nu$.
\end{enumerate}

We now state and prove the technical lemma
used in the proof of Theorem~\ref{thm:borell-poisson}.
\begin{lemma}\label{lem:tech}
	Let $G \colon [0,T] \times \NN \to \RR_+$ and assume
	that $G$ is continuous in the first variable and bounded.
	Then there exists a predictable, bounded, non-negative process $(\lambda_t)_{0 \leq t\leq T}$ satisfying
	\[
	\lambda_t = G ( t , X^\lambda_{t} ) ,
	\]
	for almost every $t\leq T$ and with probability one.
\end{lemma}
\begin{proof}
	Consider the map
	\[
	H \colon (\lambda_t)_{0 \leq t\leq T} \mapsto (G(t , X^\lambda_{t} ))_{0 \leq t\leq T}
	\]
	from the set of predictable, non-negative processes
	to itself. 
	We will  show that $H$ has a fixed point.
	Let $\lambda$ and $\mu$ be two processes in the domain of $H$.
		Since $G$ is bounded, there is
	a constant $C > 0$ such that
	\[
	\EE [ \vert G(t , X^\lambda_{t} ) - G(t , X^\mu_{t} ) \vert ]
	\leq C \, \PP ( X^\lambda_{t} \neq X^\mu_{t} ).
	\]
	The probability that the integer-valued random variable $X^\lambda_{t}$ differs from $X^\mu_{t}$
	is dominated by $\EE [ \vert X_{t}^{\lambda} -X_{t}^{\mu} \vert ] $, 
	which in turn is the average number of atoms of $N$ between
	the graphs of $\lambda$ and of $\mu$ on $[0,t)$. Since
	$\lambda$ and $\mu$ are predictable, it follows from Lemma \ref{lem_1532} that
	\[ \EE |X_{t}^{\lambda} -
	X_{t}^{\mu} | =
	\EE \left[ \int_0^t \vert \lambda_s - \mu_s \vert \, ds \right] .
	\]
	Therefore
	\[
	\EE \left[ \vert G(t , X^\lambda_{t} ) - G(t , X^\mu_{t} ) \vert \right]
	\leq C \, \EE \left[  \int_0^t \vert \lambda_s - \mu_s \vert \, ds \right] .
	\]
	This easily implies that $H$ is Lipschitz with constant $1/2$ for the distance
	\[
	d ( \lambda , \mu ) =  \int_0^T \e^{-2Ct} \, \EE [ \vert \lambda_t - \mu_t \vert ] \, dt   .
	\]
	Note that $d(\lambda, \mu) = 0$  if and only if $\lambda_t = \mu_t$ for almost every $0 \leq t \leq T$ and with probability one.
	The space of all predictable, non-negative processes $(\lambda_t)_{0 \leq t \leq T}$ with $\int_0^T \EE |\lambda_t| dt < \infty$ is complete with respect to the distance $d$. 
	Thus, being a contraction, the map $H$ has a fixed point. This fixed point is necessarily a bounded process, since $G$ is bounded. This completes the proof.
\end{proof}
\section{A discrete Pr\'ekopa-Leindler
	inequality}
\label{sec5}
Following Borell, in this section we derive a Pr\'ekopa-Leindler type inequality
from the representation formula~\eqref{eq:borell-poisson}.
Recall that if $x$ is a real number we denote its
lower integer part by $\lfloor x \rfloor$
and its upper integer part by $\lceil x \rceil$.
For $a,b \in \RR$ we denote $a \wedge b = \min \{ a, b \}$
and $a \vee b = \max \{ a, b \}$.
Recall that $\pi_T$
denotes the Poisson distribution with parameter $T$.
\begin{proposition} Let $T > 0$ and let $f,g,h,k \colon \NN \to \RR$ satisfy
	\[
	f (x) + g (y) \leq h \left( \left \lfloor \frac {x+y} 2 \right \rfloor \right)
	+ k \left( \left \lceil \frac{x+y} 2 \right \rceil \right) ,
	\quad \forall x,y\in \NN.
	\]
	Then,
	\[
	\int_\NN \e^f \, d\pi_T  \int_\NN \e^g\, d\pi_T
	\leq   \int_\NN \e^h \, d\pi_T  \int_\NN \e^k \, d\pi_T.
	\]
\end{proposition}
\begin{proof}
	By approximation we may assume
	that all
	four functions are bounded. Let $\alpha$ and $\beta$ be
	two non-negative, bounded, predictable processes.
	It follows from formula (\ref{eq:defXlambda}) that  $\lfloor ( X^\alpha + X^\beta )/2 \rfloor$
	coincides with the process $X^\lambda$, where
	\begin{equation} \label{eq_543bis}
	\lambda =  ( \alpha \wedge \beta ) \, \chi +  (\alpha \vee  \beta ) \, (1-\chi),
	\end{equation}
	and $\chi_t$ is the indicator function of the event that $X^{\alpha}_{t} +X^{\beta}_{t}$
	is even. Indeed, we see from the definition (\ref{eq:defXlambda})
	that 
	an atom of the Poisson process $N$ that lies below the graph of $\alpha \wedge \beta$ corresponds to a jump both in  $X^{\alpha}$ and in $X^{\beta}$,
	and consequently it entails a jump in $\lfloor ( X^\alpha + X^\beta )/2 \rfloor$. On the other hand, an atom of $N$ that lies between the graphs of $\alpha \wedge \beta$ and $\alpha \vee \beta$ causes a jump only if $\chi_t = 0$. This explains formula (\ref{eq_543bis}). Since $X^{\alpha}$ and $X^{\beta}$ are adapted and left-continuous in $t$, the same applies to $\chi$. Consequently $\chi$ and $\lambda$ are predictable. 	
	
	\medskip 		Similarly
	$\lceil (  X^{\alpha} +  X^{\beta} ) / 2 \rceil = X^\mu$, where
	\[
	\mu = (\alpha \wedge \beta)  \, (1-\chi) + (\alpha \vee  \beta) \, \chi .
	\]
	Note that for every $t\in [0,T]$ either $\mu_t = \alpha_t$ and $\lambda_t = \beta_t$
	or the other way around. In particular, for every function $\vphi: [0, \infty) \rightarrow \RR$,
	\[
	\vphi ( \alpha_t ) + \vphi ( \beta_t )
	= \vphi ( \lambda_t ) + \vphi ( \mu_t ) , \quad \forall t \in [0,T].
	\]
	Using the hypothesis made on $f,g,h,k$ we get that for a continuous function $\vphi$,
	\[
	\begin{split}
	f ( X^\alpha_T ) & +  g( X^\beta_T )
	- \int_0^T \vphi ( \alpha_t  )  \, dt
	- \int_0^T \vphi ( \beta_t ) \, dt \\
	& \leq
	h ( X^\lambda_T ) + k( X^\mu_T )
	- \int_0^T \vphi ( \lambda_t )  \, dt
	- \int_0^T \vphi ( \mu_t ) \, dt.
	\end{split}
	\]
	Choosing $\vphi (x) = x \log x + x -1$,
	taking expectation, and using the representation
	formula in Theorem \ref{thm:borell-poisson} for $h$ and $k$ we obtain
	\[
	\begin{split}
	\EE & \left[ f ( X^\alpha_T )  - \int_0^T \vphi ( \alpha_t  )  \, dt  \right]
	+ \EE \left[ g( X^\beta_T ) - \int_0^T \vphi ( \beta_t ) \, dt \right] \\
	& \leq \log \left( \int_\NN \e^h \, d\pi_T \right)
	+ \log \left( \int_\NN \e^k \, d\pi_T \right) .
	\end{split}
	\]
	Taking the supremum in $\alpha$ and $\beta$
	and using the representation formula for $f$ and $g$
	yields the result.
\end{proof}
Rescaling appropriately, we obtain as a corollary
a Pr\'ekopa-Leindler type inequality for the counting measure on $\mathbb Z$.
\begin{proof}[Proof of Theorem \ref{thm_603}] We may assume that all four sums in (\ref{eq_558}) are finite.
	Let $Y_n$ be a random variable having the Poisson
	law with  parameter $n$ and let $X_n = Y_n - n$.
	Applying the previous proposition to the functions
	$f,g,h,k$ (translated by $-n$) we get
	\begin{equation}\label{eq:prekopa-step}
	\EE \left[ \e^{f(X_n)} \right] \, \EE \left[ \e^{g(X_n)} \right]
	\leq \EE \left[ \e^{h(X_n)} \right] \, \EE \left[ \e^{k(X_n)} \right] .
	\end{equation}
	On the other hand, for any fixed $k\in\ZZ$,
	letting $n$ tend to $+\infty$ and using the Stirling formula we get
	\[
	\PP ( X_n = k ) = \frac { n^{n+k} }{ (n+k)! } \e^{-n} = \frac 1{\sqrt{2\pi n}} ( 1 + o (1) ).
	\]
	Hence by the dominated convergence theorem
	\[
	\sqrt{2 \pi n} \cdot \EE \left[ \e^{f(X_n)} \right] \stackrel{n \rightarrow \infty}\longrightarrow  \sum_{x\in \ZZ} \e^{f(x)},
	\]
	and similarly for $g,h,k$.
	Therefore multiplying~\eqref{eq:prekopa-step} by $n$ and letting
	$n$ tend to $+\infty$ yields the result.
\end{proof}

\section{Appendix: Proof of Lemma \ref{lem_1532}}
We write $\cB ( [0,T] ) $ for the Borel $\sigma$-field of $[0,T]$.
Let $\mu^+$ and $\mu^-$ be the measures on
$[0,T]\times \RR_+ \times \Omega$ equipped with the
$\sigma$-field $\cB([0, T]) \otimes \cB(\RR^+) \otimes \cF$
defined by
\[
\begin{split}
\mu^+ ( dt,du,d\omega ) &  = N ( \omega ) ( dt , du ) \, \PP ( d \omega ) \\
\mu^- ( dt, du , d\omega ) & = \mathcal L ( dt , du ) \,  \PP ( d\omega ) ,
\end{split}
\]
where $\mathcal L$ is the Lebesgue measure on the strip $[0, T] \times \RR_+$
while  $N(\omega)$ is the discrete measure on this strip given by the Poisson process $N$.
Let $\mathcal I$ be the $\sigma$-field on $[0,T]\times\RR_+\times\Omega$
generated by the class
\[
\mathcal J
= \{ E \times A ; \; E  \in \cB( [0,T]\times \RR_+ ) , \, A \in \mathcal F_{E^c} \}
\]
where $E^c$ denotes the complement of $E$.
We claim that $\mu^+$ and $\mu^-$ coincide on $\mathcal I$.
This is in fact the statement of Theorem~1 in~\cite{picard}, 
we recall the short proof here for completeness.
Since $\mathcal J$ is a $\pi$-system (i.e. stable by finite intersections)
and $\sigma ( \mathcal J ) = \mathcal I$, it is enough to prove
that $\mu_+$ and $\mu_-$ coincide on $\mathcal J$. On the other 
hand, if $E\times A \in \mathcal J$, the random variable $N(E)$ 
is independent of the set $A$, hence
\[
\mu^+ ( E \times A ) = \EE [ N(E ) \mathbbm 1_A ] =
\EE [ N ( E ) ] \PP (A) = \mathcal L ( E ) \PP (A) = \mu^- ( E \times A ) .
\]
Next recall the definition of the predictable $\sigma$-field $\mathcal P$
and observe that
\[
\mathcal B ( \RR_+ ) \otimes \mathcal P \subseteq \mathcal I .
\]
As a result, since $(H_t)$ and $(\lambda_t)$ are predictable,
as a function of $(t,u,\omega)$,
\[
H_t \, \mathbbm 1_{\{ u\leq \lambda_t \}}
\]
is measurable with respect to $\mathcal I$. We may therefore
integrate $H_t \, \mathbbm 1_{\{ u\leq \lambda_t \}}$
with respect to $\mu^+$ or $\mu^-$ and obtain the same outcome. In other words,
\begin{equation}
\EE \left[ \int_{ [0,T]\times \RR_+ }
H_t \, \mathbbm 1_{\{ u\leq \lambda_t \}}\, N(dt,du) \right]
= \EE \left[ \int_{ [0,T]\times \RR_+ }
H_t \, \mathbbm 1_{\{ u\leq \lambda_t \}} \, dt du  \right].
\label{eq_1557}
\end{equation}
From (\ref{eq_1557}) we obtain that
\[
\begin{split}
\EE \left[ \int_0^T H_t X^{\lambda}(dt) \right]
& = \EE \left[\int_{[0,T]\times\RR_+} H_t \mathbbm 1_{\{u\leq \lambda_t\}} \, N(dt,du) \right] \\
& = \EE \left[ \int_{[0,T]\times\RR_+} H_t \mathbbm 1_{\{u\leq \lambda_t\}} \, dt du \right] \\
& = \EE \left[ \int_0^T H_t \mathbbm \lambda_t \, dt \right],
\end{split}
\]
completing the proof of Lemma~\ref{lem_1532}.

\bigskip

\smallbreak
\noindent Department of Mathematics, Weizmann Institute of Science, Rehovot 76100 Israel, and
School of Mathematical Sciences, Tel Aviv University, Tel Aviv 69978 Israel.

\smallbreak
\hfill \verb"boaz.klartag@weizmann.ac.il"

\bigskip
\noindent
CEREMADE (UMR CNRS 7534),
Universit\'e Paris-Dauphine, 75016 Paris, France, and 
D\'epartement de Math\'ematiques et Applications (UMR CNRS 8553), \'Ecole Normale Sup\'erieure, 75005, Paris, France.

\smallbreak
\hfill \verb"lehec@ceremade.dauphine.fr"

\end{document}